\documentclass{amsart}
\usepackage{amssymb,latexsym}
\theoremstyle{plain}
\newtheorem*{thm*}{Theorem}
\newtheorem{thm}{Theorem}
\newtheorem{cor}{Corollary}
\newtheorem{prop}{Proposition}
\newtheorem{lemma}{Lemma}

\newtheorem{question}{Question}

\newcommand{\dto}{\dashrightarrow}
\newcommand{\bb}[1]{\mathbb{#1}}
\newcommand{\vol}{\textrm{vol}}
\newcommand{\m}[1]{\mathcal{#1}}

\begin{document}

\title[Weighted hypersurfaces]{Weighted hypersurfaces with either assigned volume or many vanishing plurigenera}
\author{E. Ballico, R. Pignatelli, L. Tasin}

\address{Department of Mathematics\\
University of Trento\\
Via Sommarive 14\\
38123 Povo (TN), Italy}
\email{ballico@science.unitn.it, pignatel@science.unitn.it, tasin@science.unitn.it}
\thanks{The authors are partially supported by MIUR and GNSAGA of INdAM (Italy).}

\begin{abstract}
In this paper we construct, for every $n$, smooth varieties of general type of dimension $n$ with the first 
$\lfloor \frac{n-2}{3} \rfloor$ plurigenera equal to zero. 
Hacon-McKernan, Takayama and Tsuji have recently shown that there are numbers $r_n$ such that $\forall r\ge r_n$,
the $r-$canonical map of every variety of general type of dimension $n$ is birational. 
Our examples show that $r_n$ grows at least quadratically as a function of  $n$. Moreover they show that the 
minimal volume of a variety of general type of dimension $n$ is smaller than $\frac{3^{n+1}}{(n-1)^{n}}$.

In addition we prove that for every positive rational number 
$q$ there are smooth varieties of general type with volume $q$ and dimension arbitrarily big. 
\end{abstract}

\maketitle

\section{Introduction}

Let $X$ be a  smooth variety of general type (we always intend projective over $\bb C$). 
Since the canonical divisor $K_X$ is intrinsically associated to $X$, the study of the pluricanonical 
systems $|rK_X|$, of the induced maps $\phi_r$, and of the canonical ring $R(X):=\oplus H^0(X,mK_X)$ 
is a classical and important matter. 
Further these objects are birational invariants.

It is natural to ask how "small" can the Hilbert function of $R(X)$ be. 
There are many different possible definition of "small"; 
we are mainly interested in two of them. First, we would like to consider varieties of small 
canonical volume. Recall that $vol(X):=\limsup_{m \rightarrow +\infty} \frac{n!h^0(X,mK_X)}{m^n}$, so making the volume small is the same as 
making small the asymptotical behaviour of the Hilbert function. Second, we would like to 
understand which plurigenera $P_m := h^0(X,m K_X )$ may be zero. 

For curves of general type $vol(X) \ge 2$ and $K_X$ is effective; for surfaces $vol(X) \ge 1$ and 
$P_2 \neq 0$, while there are surfaces of general type with $P_1=0$.
For threefolds the record from both points of view is attained by an example of Iano-Fletcher 
(see 15.1 of \cite{If}) with volume $\frac{1}{420}$ 
and $P_1=P_2=P_3=0$. See also \cite{chensquare1} and \cite{chensquare2} for related results in dimension $3$.

In higher dimension we are able to prove the following.

\begin{thm}\label{thm3}
Let $n\ge 5$ be an integer.  There exists a smooth variety of general type $X$ of dimension $n$ such 
that $H^0(X, mK_X)=0$ for $0<m< \lfloor \frac{n+1}{3} \rfloor$ and $\vol(X)<\frac{3^{n+1}}{(n-1)^{n}}$.
\end{thm}

\hspace{1em}

Recently the following result, generalisation of a famous theorem of \cite{bo}, was proven in \cite{HM}, \cite{Tak} and \cite{Tsu}.

\begin{thm*}
For any positive integer $n$ there exists an integer $r_n$ such that if $X$ is a smooth variety of general type and dimension $n$, then 
$$
\phi_{r}: X \dto \bb P(H^0(X, r K_X))
$$
is birational onto its image for all $r\geq r_n$.
\end{thm*}

A consequence of this theorem is that for any smooth variety of general type $X$ of dimension $n$ we have
$$
\vol (K_X) \geq \frac{1}{r_n^n}.
$$
and that the set of the volumes of the manifolds of general type of dimension $n$ has a minimum $v_n>0$.

An instant consequence of Theorem \ref{thm3} is the following

\begin{cor}
Let $v_n$ be the minimal volume of an $n$-dimensional smooth variety of general type. 
Then $\lim_{n \to \infty} v_n =0$.
\end{cor}

\hspace{1em}

A natural problem is to estimate  $r_n$. It is well known that $r_1=3$ and $r_2=5$.
By the mentioned example of Iano-Fletcher we know that $r_3 \geq 27$
(see 15.1 of \cite{If}). Let $x'_n$ be the minimal positive integer such that for every 
$n$-dimensional smooth variety $X$ of general type there is an integer 
$t\le x'_n$ such that $\phi_t$ is generically finite; obviously $r_n \ge x_n'$.

The examples of Theorem \ref{thm3} provide a lower bound for $x_n'$ 
(and therefore for $r_n$) which is quadratic in $n$. More precisely

\begin{thm}\label{thm4}
For any integer $n \ge 7$ we have 
$$
r_n \ge x'_n \ge \frac{n(n-3)}9,
$$
In particular
$$
{\lim _{n \to +\infty} r_n}={\lim _{n \to +\infty} x'_n}=+\infty
$$
\end{thm}

The canonical system of these varieties is not ample. In view of Fujita's conjecture, 
smooth varieties with ample canonical system should not give anything better than a 
linear bound. We show
\begin{thm}\label{thm}
For any positive integer $n>0$  there is a smooth variety $X$ of dimension $n$ such that $K_X$ is ample and $\phi_{|tK_X|}$ is not birational for $t<n+3$ if $n$ is even or $t<n+2$ if $n$ is odd.

\end{thm}

The idea of this example is taken from \cite{Ka}, Example 3.1 (2). 

These bounds are optimal up to dimension 3 (for the three dimensional case see \cite{CCZ}).  
Note that the bound for $n$ even is the same predicted by Fujita's conjecture for the very ampleness, 
while for $n$ odd is one less. 

\hspace{1em}

Let $r'_n$ be the minimal positive integer such that for every $n$-dimensional smooth variety $X$ of general type there is an integer $r\le r'_n$ such that 
$\vert rK_X\vert$ induces a birational map. Let $x_n$ be the minimal integer such that for every $n$-dimensional smooth variety of general type and 
every integer $t \ge x_n$  the map induced by $| tK_X|$ is generically finite. Of course 
$r_n \ge r'_n \ge x'_n$,  $r_n \ge x_n \ge x'_n$. It is also natural to study the behaviour of these numbers.

Taking $X = Y\times C$ with $C$ a smooth curve of genus 
$2$ and $Y$ a smooth variety of general type we get $r_n\ge r_{n-1}$ and $r'_n \ge r'_{n-1}$ for all $n\ge 2$. 

\begin{question}\label{a2}
Is $r_{n+1} > r_n$ for all $n$? Is $r'_{n+1} > r'_n$ for all $n$?
\end{question}

All our examples satisfy $h^i(X,\mathcal {O}_X)=0$
for all $1 \le i \le n-1$.  Manifolds $X$ with the property $h^i(X,\mathcal {O}_X)=0$
for all $1 \le i \le n-1$ are very special, but it may be worthwhile to fix the integer 
$q:= h^1(X,\mathcal {O}_X)$ and study the integers $r_n(q)$, $r'_n(q)$, $x_n(q)$, $x'_n(q)$ 
 ($r_n(\ge q)$, $r'_n(\ge q)$, $x_n( \ge q)$, $x'_n( \ge q)$  and 
$r_n(\le q)$, $r'_n(\le q)$, $x_n( \le q)$, $x'_n( \le q)$) 
 obtained taking only manifolds with irregularity $q$ (resp. $\ge q$, resp. $\le q$). 
Taking $X = Y\times D$ with $D$ a curve of genus $x\ge 2$ we get
$r_n(q) \ge r_{n-1}(q-x)$ for all integers $n\ge 2$ and $q, x$ such that $2 \le x \le q$ 
(and similarly  for the other integers $r_n'(q)$). 

\hspace{1em}

Our last result is that every positive rational number is a canonical volume:

\begin{thm} \label{thm2}
Let $q=r/s$ be a rational number with $r,s >0$ and $(r,s)=1$. There are infinite positive integers $n$ such that there is a smooth variety of dimension $n$ with
$$
\vol (X)= \frac{r}{s}.
$$
\end{thm}
 
This research started from a question of R. Ghiloni on the asymptotical behaviour of $r_n$: we thank him heartily.

\section{The proofs}

We will need the following lemma.

\begin{lemma}
Let $\bb P = \bb P(a_0, \ldots, a_n)$ be a well-formed weighted projective space.  If the coordinate 
points are canonical singularities then all the singularities of $\bb P$ are canonical.
\end{lemma}

\begin{proof}
Let $P$ be a singular point of $\bb P$. Define $U :=\{0,1, \ldots, n  \}$ and let $S \subset U$ be the 
subset of the variables nonzero at $P$. By hypothesis, if $\#S=1$, then $P$ is a canonical singularity. 
Assume then $\#S >1$.

We denote by $h_S$ the highest common factor of the set $\{a_i | i \in S\}$.
Then $h_S>1$ and, chosen a $k \in S$, 
$P$ is a cyclic quotient singularity of type 
$$
\frac{1}{h_S} (a_0,\ldots,\hat{a}_k,\ldots,a_n).
$$

By the criterium of \cite{Re}, page 376, $P$ is a canonical singularity if and only if
$$
\frac{1}{h_S}\sum_{i=0}^n \overline{ja_i}^S \ge 1 \quad \quad \mbox{for all }  1 \le j \le h_S -1
$$
where $\overline{a}^S$ denotes the smallest (non negative) residue of $a$ $\textrm{mod}\ h_S$.

We argue by contradiction. Suppose 
$$
\frac{1}{h_S}\sum_{i=0}^n \overline{ja_i}^S < 1 \quad \quad \mbox{for some }  1 \le j \le h_S -1.
$$

Take a $k \in S$. Then $a_k=mh_S$ for a positive integer $m$. Note that $mj < a_k$. Then we have

$$
\frac{1}{a_k}\sum_{i=0}^n \overline{mja_i}^{\{k\}}= \frac{1}{m h_S}\sum_{i=0}^n \overline{mja_i}^{\{k\}} =\frac{1}{h_S}\sum_{i=0}^n \overline{ja_i}^S < 1,
$$
which contradicts the hypothesis. 
\end{proof}

\begin{prop}\label{prop}
Let $k \ge 2$ and $l \ge 0$ be  integers. Consider the weighted projective space 
$$
\bb P := \bb P(k^{(k+2)}, (k+1)^{(2k-1)}, (k(k+1))^{(l)}).
$$

Then the general hypersurface $X_d$ in $\bb P$ of degree $d:=(l+3)k(k+1)$ 
has at worst canonical singularities, $K_{X_d} \sim \m O_{X_d}(1)$, $\dim X_d=3k+l-1$ and 
$$
\vol (X_d) = \frac{(l+3)}{k^{k+1+l}(k+1)^{2k-2+l}}.
$$
\end{prop}

\begin{proof}
The weighted projective space $\bb P$ is well-formed since $k \ge 2$. 
 We use the criterium of \cite{Re}, page 376, to control that the singularities of $\bb P$ are 
canonical. By the previous lemma it is enough to look at the coordinates points. They are of three types.

\begin{enumerate}
	\item  The singularities of type $$\frac{1}{k} (k^{(k+1)}, (k+1)^{(2k-1)}, (k(k+1))^{(l)}).$$ We have to check that
	$$
	\frac{1}{k}(2k-1)\overline{j(k+1)}\ge 1
	$$  
	for $1 \le j \le k-1$, where $\bar{\ }$ denotes the smallest (non negative) residue $\textrm{mod}\ k$. This is trivial since $\overline{j(k+1)}=j \ge 1$ for $1 \le j \le k-1$.
	
	\item   The singularities of type $$\frac{1}{k+1} (k^{(k+2)}, (k+1)^{(2k-2)}, (k(k+1))^{(l)}).$$ We have to check that
	$$
	\frac{1}{k+1}(k+2)\overline{jk}\ge 1
	$$  
	for $1 \le j \le k$, where $\bar{\ }$ denotes the smallest (non negative) residue $\textrm{mod}\ k+1$. This is trivial since $\overline{jk}\ge 1$ for $1 \le j \le k$.

	\item The singularities of type $$\frac{1}{k(k+1)} (k^{(k+2)}, (k+1)^{(2k-1)}, (k(k+1))^{(l-1)}),$$ 
 when $l \ge 1$. We have to check that
	$$
	\frac{1}{k(k+1)}((k+2)\overline{ik} + (2k-1)\overline{i(k+1)})\ge 1 
	$$  
	for $1 \le i \le k(k+1)-1$, where $\bar{\ }$ denotes the smallest (non negative) residue $\textrm{mod}\ k(k+1)$. 
This follows since $k \not | j$ then $\overline{j(k+1)} \ge k+1$ and if $(k+1) \not | j$ then 
$\overline{jk} \ge k$.
\end{enumerate}

Now note that $\m O_{\bb P}(d)$ is base point free 
(since $d$ is a multiple of every weight) and locally free (by Lemma 1.3 of \cite{Mo}).

Then we can apply a Koll\'ar-Bertini theorem (Proposition 7.7 of \cite{Ko}, see also Theorem 1.3 of \cite{ReiCan}) 
to conclude that the general hypersurface $X_d$ of degree $d$ is canonical 
(and obviously well-formed and quasi-smooth, cf. \cite{If}).

Finally, by adjunction (6.14 of \cite{If}), $K_{X_d} \sim \m O_X(1)$ and so

$$
\vol (X_d) =\frac{(l+3)k(k+1)}{k^{k+2}(k+1)^{2k-1}(k(k+1))^l}= \frac{(l+3)}{k^{k+1+l}(k+1)^{2k-2+l}}.
$$
\end{proof}

\begin{proof}[Proof of theorem \ref{thm3}]
Write $n= 3k+l-1$, with integers $k \ge 2$ and $0\le l \le 2$, so $k=\lfloor \frac{n+1}{3} \rfloor$.  We can apply the previous proposition to obtain a canonical variety $X_d$ of dimension $n$ in the projective space
$$
\bb P := \bb P(k^{(k+2)}, (k+1)^{(2k-1)}, (k(k+1))^{(l)}).
$$

By Theorem 3.4.4  (and the proof of the lemma above) in \cite{Do} we deduce that
$$
H^0(X, mK_{X_d})=0
$$
for $ 0<m <k$.

Moreover
$$
\vol(X_d)= \frac{(l+3)}{k^{k+1+l}(k+1)^{2k-2+l}} < 
\frac{(l+3)}{k^{n+l}} =
\frac{l+3}{3k^l} \cdot \frac{3}{k^{n}} \le \frac{3^{n+1}}{(n-1)^{n}}.
$$

Take as $X$ any desingularization of $X_{d}$.
\end{proof}

\begin{proof}[Proof of theorem \ref{thm4}]
Let $n= 3k+l-1$, with integers $k \ge 2$ and $2\le l \le 4$ so $k=\lfloor \frac{n-1}{3} \rfloor$. 
 We can apply the  proposition \ref{prop} to obtain a canonical variety $X_d$ of dimension $n$ in the projective space
$$
\bb P := \bb P(k^{(k+2)}, (k+1)^{(2k-1)}, (k(k+1))^{(l)}).
$$
We denote the coordinates of this space $z_i$ for $0 \le i \le n+1$.
Recall that $K_{X_d} \sim  \m O_{X_d} (1)$.

Fix $t < k(k+1)$. Thanks to Proposition 3.3 of \cite{Mo} we get
$$
|\m O_X(t)|=|\m O_{\bb P} (t)|,
$$
but the last two variable $z_n$ and $z_{n+1}$ can't appear in an element of $|\m O_{\bb P} (t)|$ 
for degree's reasons.   

Take as $X$ any desingularization of $X_{d}$.
Then the map induced by $|t K_X|$ is not generically finite. In particular 
$$x'_n\ge k(k+1) \ge \frac{n(n-3)}9.$$
\end{proof}

\begin{proof}[Proof of theorem \ref{thm}]
We first consider  the case $n$ even.  Let $d=n+3$ and let $\bb P$ the weighted projective space $\bb P(1^{(n)}, 2, d)$. The coherent sheaf $\mathcal {O}_{\mathbb {P}}(2d)$ is a line bundle
(\cite{Mo}, Lemma 1.5). We call the coordinates of this space $z_i$ for $0 \le i \le n+1$.  Since $d$ is odd the general weighted hypersurface $X$ of degree $2d$ do not meet the singularities of $\bb P$ and it is smooth. Note that its equation is of the form 
$$z_{n+1}^2= P(z_0, \dots, z_n)$$
where $P$ is a polynomial of weighted degree $2d$. Moreover we have $K_X \sim  \m O_X (1)$.

If we take a positive integer $t < d$, then the linear system $|t K_X|$ does not induce a birational map. Indeed by Proposition 3.3 of \cite{Mo} we have
$$
|\m O_X(t)|=|\m O_{\bb P} (t)|,
$$
but the variable $z_{n+1}$ can't appear in an element of $|\m O_{\bb P} (t)|$ for degree's reasons and so the induced map has at least degree 2.

If $n=1$ we use a smooth curve of genus $2$. If $n$ is odd and $n\ge 3$ we consider a manifold $Y$ 
of dimension $n-1$ such that $r_n \ge n+2$ and define $X:= Y \times C$ where $C$ is smooth curve of 
genus $2$. Alternatively, for $n$ odd, take $d=n+2$ and a general hypersurface of degree $2d$ in $\bb P(1^{(n+1)},d)$.
\end{proof}

\begin{proof}[Proof of theorem \ref{thm2}]
Let $b$ be a positive integer such that
$$
br \equiv 1 \quad mod\ s,
$$
and write
$$
br-1=ts. 
$$
Let $a$ be a positive integer such that $(a,s)=(a,b)=1$.

We set
$$
n:=rab+1-a-s-b, \quad \quad d:=n-1+a+s+b=rab
$$
$$
\bb P= \bb P(1^{(n-2)},a, s,b).
$$
Observe that choosing $a$ and $b$  we can have $n$ arbitrarily large. Hence we may assume $n\ge 3$.

Note that $(a,s)=(a,b)=(s,b)=1$, hence the only singularities of $\bb P$ are $P=(0^{(n-2)},0,1,0)$, $Q=(0^{(n-2)},1,0,0)$ and $R=(0^{(n-2)},0,0,1)$.

Now consider a general weighted hypersurface $X_d$ of degree $d$ in $\bb P$.  The sheaf $\mathcal {O}_{\bb P}(1)$ is locally free and spanned
outside $P$. By Bertini's theorem applied to $\bb P \setminus \{P\}$ we get that $X_d$ is smooth outside $P$. Since $s \not | d$, $a  | d$ and $b | d$, so $P \in X_d$ while $Q,R \notin X_d$. Since
$n\ge 3$ and $X_d$ has a unique singular point, it is well-formed in the sense of \cite{If}.

We will show that for $a$ and $b$ large enough, $X_d$ is a (well-formed) quasi-smooth variety with at most a terminal singularity in $P$. Then we would have finished. Indeed note that by adjunction (6.14 of \cite{If}) $K_X \sim \mathcal O_X(1)$ is ample and therefore

$$
\vol (X_d)= K_X ^n = \mathcal O_X(1)^n=  \frac{d}{asb}=\frac{r}{s}.
$$

To control the quasi-smoothness we use  criterium 8.1 of \cite{If}.

Let $z_0, \ldots,  z_{n+1}$ be the coordinates of $\bb P$. For every $I \subset \{0,\ldots, n+1\}$ except $I=\{n\}$ the condition 2.a of 8.1 in \cite{If} is satisfied because the only variable whose degree doesn't divide $d$ is $z_n$.  In the case $I=\{n\}$ we can use 2.b since $d=tas + a$ and we can take the monomial 
$$
z_n^{ta} z_{n-1}.
$$
 
Since $X_d$ is quasi-smooth its singularities are induced by those of $\bb P$ and so we have only  to control that $X_d$ is terminal in $P$.

Let $f=0$ be an equation of $X_d$. We can write
$$
f= z_n^{ta} z_{n-1} + \ldots
$$

We consider the affine piece $(z_n=1)$. The point $P \in X_d$ looks like 
$$
(\tilde f =f(z_0, \ldots,z_{n-1},1,z_{n+1})= z_{n-1}+\ldots =0) \subset \bb A^{n+1} / \epsilon
$$
where $\epsilon$ is a primitive $s$-th root of unity and acts via
$$
z_i \mapsto \epsilon z_i \quad  0\le i \le n-2,
$$
$$
z_{n-1} \mapsto \epsilon^a z_{n-1}
$$
and
$$
z_{n+1} \mapsto \epsilon^b z_{n+1}.
$$

Note that  $\partial f / \partial z_{n-1} \neq 0$ in $P$, hence, by the Inverse Function Theorem, $z_i$ are local coordinates for $P$ in $X_d$ for $i\neq n-1,n$. This gives a quotient singularity of type
$$
\frac{1}{s} (1^{(n-2)},b).
$$
By the criterium of \cite{Re}, page 376, if $n-2 \ge s$ then $P$ is terminal.

Now you can simply take a desingularization of $X_d$.
\end{proof}

\providecommand{\bysame}{\leavevmode\hbox to3em{\hrulefill}\thinspace}

\end{document}